\documentclass[12pt]{article}  
\usepackage{theorem}                 
\usepackage{amssymb}

\usepackage{graphicx}
\newtheorem{theorem}{Theorem}
\newtheorem{corollary}{Corollary}
\newtheorem{proposition}{Proposition}

\input{cyracc.def}

\newfont{\cyrfnt}{wncyr10}
\newcommand{\cyr}{\baselineskip12.5pt\cyrfnt\cyracc}

\theorembodyfont{\rmfamily}
\newtheorem{remark}{Remark}

\newenvironment{proof}
{\noindent{\em Proof\/}.}{{ $\Box$}\smallskip\par}

\newcommand{\CC}{{\Bbb C}}

\newcommand{\ZZ}{{\Bbb Z}}

\newcommand{\calO}{{\widetilde{\cal O}}}

\newcommand{\calC}{{\cal C}}

\newcommand{\calD}{{\widetilde{\cal D}}}
\newcommand{\calT}{{\widetilde{\cal T}}}

\newcommand{\eps}{\varepsilon}

\title{Poincar\'e series and monodromy of the simple and unimodal boundary singularities}
\author{Wolfgang Ebeling
\thanks{Supported by  INTAS-05-7805. AMS Subject classification: 32S40, 58K10.
}
}

\date{}

\begin{document}

\maketitle

\begin{abstract} A boundary singularity is a singularity of a function on a manifold with boundary. The simple and unimodal boundary singularities were classified by V.~I.~Arnold and V.~I.~Matov. The McKay correspondence can be generalized to the simple boundary singularities. We consider the monodromy of the simple, parabolic, and exceptional unimodal boundary singularities. We show that  the characteristic polynomial of the monodromy is related to the Poincar\'e series of the coordinate algebra of the ambient singularity.
\end{abstract}

\section*{Introduction}
There is a well known classification of hypersurface singularities due to V.~I.~Ar\-nold. The first classes are the simple and the unimodal singularities. The simple surface singularities in $\CC^3$ are precisely the Kleinian singularities which are quotients of $\CC^2$ by  finite subgroups of ${\rm SU}(2)$. These are weighted homogeneous singularities. Also the parabolic and the 14 exceptional unimodal singularities can be defined by weighted homogeneous equations. To such a singularity, there can be associated two invariants of different nature. One is the Poincar\'e series of the (graded) coordinate algebra of the singularity. The other one is the characteristic polynomial of the monodromy operator of the singularity. It turned out that there is a close relation between these two invariants. More precisely, it was shown that the Poincar\'e series of the simple (except $A_{2n}$) and the 14 exceptional unimodal singularities is the quotient of the characteristic polynomials of two Coxeter elements which are related to the monodromy operators of the singularities. For the simple singularities this was derived from the McKay correspondence.

Some of these singularities admit a symmetry. This leads to the boundary singularities considered by Arnold in \cite{A78}. Arnold and V.~I.~Matov classified the simple and unimodal boundary singularities. Here we investigate the relation between the characteristic polynomial of the monodromy  and the Poincar\'e series of the ambient hypersurface singularity  for such a singularity. The simple boundary singularities arise from simple hypersurface singularities and there is a generalization of the McKay correspondence for these cases. R.~Stekolshchik has extended our theorem for the Kleinian singularities to this generalized McKay correspondence. Here we give an interpretation of his result in terms of singularity theory.

For 7 of the 12 exceptional unimodal boundary singularities we show that there is a direct relation between the Poincar\'e series of the ambient singularity and the characteristic polynomial of the monodromy.

The author is grateful to Sabir~Gusein-Zade for suggesting him to study relations between Poincar\'e series and monodromy for boundary singularities.

\section{Simple and unimodal hypersurface singularities}

Let $X=\{ f(x,y,z)=0 \} \subset \CC^3$ be the zero set of a weighted homogeneous function $f$ with weights $q_1, q_2, q_3$ and degree $d$. The coordinate algebra $A=\CC[x,y,z]/(f)$ has a natural grading $A= \oplus_{m=0}^\infty A_m$ where $A_m$ consists of the polynomials which are weighted homogeneous of degree $m$. The series $p_X(t)=\sum_{m=0}^\infty \dim A_m \cdot t^m$ is called the Poincar\'e series of $X$. One has
$$p_X(t) = \frac{(1-t^d)}{(1-t^{q_1})(1-t^{q_2})(1-t^{q_3})}.$$

Assume that $0$ is an isolated singularity of $X$.
Let $B_\eps$ be a ball in $\CC^3$ around the origin of a sufficiently small radius $\eps >0$, let $\lambda \in \CC$ be sufficiently small with $|\lambda| >0$, and let $X_\lambda := f^{-1}(\lambda) \cap B_\eps$ be the Milnor fibre of the function $f$. Let $\phi_X(t)$ be the characteristic polynomial of the monodromy operator $c: H_2(X_\lambda) \to H_2(X_\lambda)$.

If $\phi$ is a rational function of the form
$$\phi(t)= \prod_{k | d} (1-t^k)^{\alpha_k} \mbox{ with } \alpha_k \in \ZZ,$$
then the {\em Saito dual} $\phi^\ast$ of $\phi$ is defined as
$$\phi^\ast(t)= \prod_{k | d} (1-t^{d/k})^{- \alpha_k}.$$

The simple and unimodal singularities of hypersurfaces in $\CC^3$ are classified by Arnold (see \cite{AGV85}) and they fall into the following classes:
\begin{itemize}
\item Simple singularities $X=A_\mu, D_\mu, E_6, E_7, E_8$.
\item Parabolic singularities $X= \widetilde{E}_6, \widetilde{E}_7, \widetilde{E}_8$.
\item Hyperbolic singularities $X=T_{p,q,r}$.
\item 14 exceptional unimodal singularities.
\end{itemize}

Except the hyperbolic singularities, all singularities are, for a certain value of the modulus in the unimodal case, given by weighted homogeneous equations.
To an exceptional unimodal hypersurface singularity $(X,0)$ in $\CC^3$ one can associate two triples of numbers. A minimal good resolution of the singularity has an exceptional divisor consisting of a rational curve of self-intersection $-1$ intersected by three rational curves of self-intersection $-b_1$, $-b_2$, and $-b_3$ respectively \cite{Dolgachev74}. The numbers $b_1,b_2,b_3$ are called the {\em Dolgachev numbers} ${\rm Dol}(X)$ of the singularity $(X,0)$. The Coxeter-Dynkin diagram with respect to a certain distinguished basis of vanishing cycles of $H_2(X_\lambda)$ has the shape of Figure~\ref{Fig1}. The numbers $p_1, p_2, p_3$ are called the {\em Gabrielov
numbers} ${\rm Gab}(X)$ of the singularity \cite{G75}. Here each vertex represents a
sphere of self-intersection number $-2$, two vertices connected by a
single solid edge have intersection number 1, and two vertices connected by a
double broken line have intersection number $-2$. The intersection number is 0 if the corresponding vertices are not connected.
\begin{figure}\centering
\unitlength1cm
\begin{picture}(8.5,7.5)
\put(1,0.5){\includegraphics{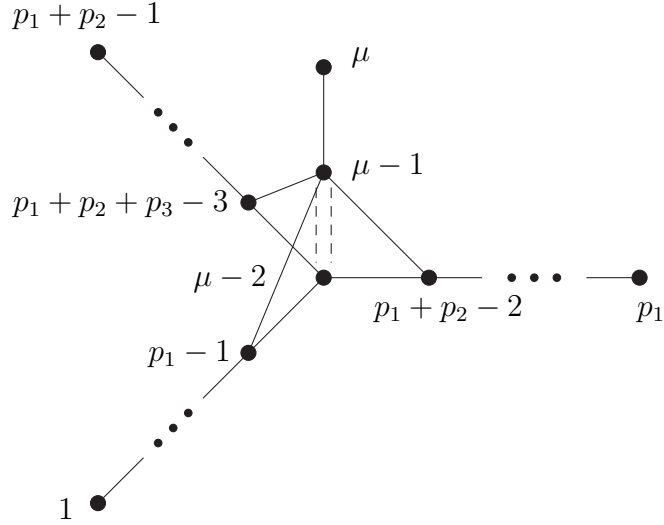}}
\put(0.6,0.4){$1$}
\put(1.8,2.5){$p_1-1$}
\put(2.4,3.5){$\mu-2$}
\put(0,4.5){$p_1+p_2+p_3-3$}
\put(0,7){$p_1+p_2-1$}
\put(4.8,3.1){$p_1+p_2-2$}
\put(8.3,3.1){$p_1$}
\put(4.5,5){$\mu-1$}
\put(4.5,6.5){$\mu$}
\end{picture}
\caption{Coxeter-Dynkin diagram of an exceptional unimodal singularity}
\label{Fig1}
\end{figure}
{\em Arnold's strange duality} is the observation that there is an involution  $X \mapsto X^\ast$ on the
set of the 14 exceptional unimodal singularities such that
$${\rm Dol}(X) = {\rm Gab}(X^\ast), \quad {\rm Gab}(X)={\rm Dol}(X^\ast).$$
Note that a singularity of type $T_{p,q,r}$ has a Coxeter-Dynkin diagram of the form of Figure~\ref{Fig1} with $p_1=p$, $p_2=q$, $p_3=r$ and the vertex with label $\mu$ omitted. The characteristic polynomial of the monodromy of a singularity of type $T_{p,q,r}$ is equal to
$$\phi_{T_{p,q,r}}(t) = (1-t)^{-1}(1-t^p)(1-t^q)(1-t^r).$$

\begin{theorem} \label{thmhyp}
The Poincar\'e series of the singularity $(X,0)$ is a quotient 
$$p_X = \frac{\phi}{\psi}$$
of two polynomials in the variable $t$ in the following cases:
\begin{itemize}
\item[{\rm (i)}] $(X,0)$ is simple. If $(X,0)$ is not of type $A_{2n}$ then $\phi=\phi_X=\phi_X^\ast$ is the characteristic polynomial of the Coxeter element of the corresponding root system and $\psi$ is the characteristic polynomial of the Coxeter element of the corresponding affine root system.
\item[{\rm (ii)}] $(X,0)$ is parabolic. In this case, $\phi=\phi_X^\ast$ and $\psi=1$.
\item[{\rm (iii)}] $(X,0)$ is an exceptional unimodal singularity. In this case, $\phi=\phi_{X^\ast}=\phi_X^\ast$ and $\psi=\phi_{T_{p,q,r}}$ where $p,q,r$ are the Dolgachev numbers of the singularity $(X,0)$.
\end{itemize}
\end{theorem}

\begin{remark} Note that in Theorem~\ref{thmhyp} (i) and (iii), the Poincar\'e series is the quotient of the characteristic polynomials of the Coxeter elements corresponding to two Coxeter-Dynkin diagrams which differ only by one vertex. In Case (i) the denominator diagram has one vertex more and in Case (iii) one vertex less than the numerator diagram.
\end{remark}

\begin{proof}
(i) is \cite[Theorem~2]{E02}, (ii) follows from \cite[Example~1]{E02}, and (iii) follows from \cite[Theorem~1]{E02} and \cite[Proposition~1 and Remark~1]{E03}. Note that (iii) also follows from results of H.~Lenzing and J.~de la Pe\~{n}a \cite{LdlP}. See also the forthcoming paper \cite{EP}.
\end{proof}

\section{McKay correspondence}

Theorem~\ref{thmhyp}(i) can be derived from the McKay correspondence, see \cite{E02}.
Let $G \subset SU(2)$ be a finite subgroup. Then $(X,0)=(\CC^2/G,0)$ is one of the simple surface singularities. If $G$ is a cyclic group of order $\mu+1$ then $(X,0)$ is of type $A_\mu$. In this case, the weights and the degree of $(X,0)$ are not unique. We assume that a singularity $(X,0)$ of type $A_\mu$  is given by the standard equation $f(x,y,z)=x^{\mu+1} + y^2 +z^2=0$.
If $G$ is not cyclic of odd order (i.e. $(X,0)$ is not of type $A_{2n}$) then we have
$$p_X(t^2) = P_G(t) = \sum_{m = 0}^\infty \dim S^m(\CC^2)^G \cdot t^m,$$
where $P_G(t)$ is the Poincar\'e series of the algebra of invariants of the group $G$.
Let $\rho_0, \ldots, \rho_n$ be the inequivalent irreducible representations of $G$ where $\rho_0$ is the trivial representation. Denote by $\rho: G \to {\rm SU}(2)$ the given natural representation. 
Define an $(n+1) \times (n+1)$-matrix $B=(b_{ij})$ by decomposing the tensor products
$$\rho_j \otimes \rho = \oplus_{i=0}^n b_{ij} \rho_i$$
into irreducible components. Define the matrix $C$ by $B=2I-C$. McKay \cite{McKay80} observed that the matrix $C$ is the affine Cartan matrix corresponding to a simply laced root system $R$. In \cite{E02} the following theorem is proved. 

\begin{theorem} \label{theoEb} 
Let $G$ be not cyclic of odd order. Then the Poincar\'e series $P_G$ satisfies
$$P_G(t) = \frac{\phi(t^2)}{\psi(t^2)}$$
where $\phi$ is the characteristic polynomial of the Coxeter element and $\psi$ is the characteristic polynomial of the affine Coxeter element of the corresponding root system $R$.
\end{theorem}

Slodowy \cite{Sl80, Sl88} has found a generalization of the McKay correspondence to include the non simply-laced Coxeter-Dynkin diagrams. We recall Slodowy's construction.

Let $H \subset {\rm SU}(2)$ be a finite subgroup and let $G \lhd H$ be a normal subgroup of $H$. If $\rho: H \to {\rm GL}(V)$ is a representation of $H$, denote by $\rho^\downarrow$ its restriction to the subgroup $G$. Let $\rho_0, \ldots, \rho_m$ be the inequivalent irreducible representations of $H$ where $\rho_0$ is the trivial representation. Let $\rho: G \to {\rm SU}(2)$ be the given natural representation.
Let $\rho^\downarrow_0, \ldots , \rho^\downarrow_n$ be the different inequivalent restrictions to $G$ of the representations $\rho_0, \ldots, \rho_m$. Define an $(n+1) \times (n+1)$-matrix $B=(b_{ij})$ by decomposing the tensor products
$$\rho^\downarrow_j \otimes \rho = \oplus_{i=0}^n b_{ij} \rho^\downarrow_i$$
into irreducible components. Define the matrix $C$ by $B=2I-C$. Then in the cases listed in Table~\ref{Tab1}, $C$ is the affine Cartan matrix corresponding to the dual $R^\vee$ of the root system $R$ in the table.

The following generalization of Theorem~\ref{theoEb} is due to R.~Stekolshchik \cite{St06}. 

\begin{theorem}[Stekolshchik]  \label{Thm:St}
Let $G$, $H$, and $R$ be as in Table~\ref{Tab1}. Then the Poincar\'e series $P_G$ satisfies
$$P_G(t) = \frac{\phi_{R^\vee}(t^2)}{\psi_{R^\vee}(t^2)}$$
where $\phi_{R^\vee}$ is the characteristic polynomial of the Coxeter element and $\psi_{R^\vee}$ is the characteristic polynomial of the affine Coxeter element of the corresponding dual root system $R^\vee$.
\end{theorem}

\begin{table}
$$
\begin{tabular}{|c|c|c|c|c|c|c|}
\hline
$G$ & $X=\CC^2/G$ & $H$ & $R$ & $p_X$ & $\psi_R$ & $\phi_R$ \\
\hline
$\calC_{2n}$ & $A_{2n-1}$ & $\calD_n$ & $B_n$ & $2n/1 \cdot n \cdot n$ & $2 \cdot (n-1)$ & $2n/n$  \\
$\calD_{n-1}$ & $D_{n+1}$ & $\calD_{2(n-1)}$ & $C_n$ & $2n/2 \cdot (n-1) \cdot n$ & $1 \cdot n$ & $2n/n$ \\
$\calT$ & $E_6$ & $\calO$ & $F_4$ & $12/3 \cdot 4 \cdot 6$  & $2 \cdot 3$ & $2 \cdot 12/4 \cdot 6$ \\
\hline
$\calD_2$ & $D_4$ & $\calT$ & $G_2$ & $6/2 \cdot 2 \cdot 3$ & $1 \cdot 2$ & $1 \cdot 6/2 \cdot 3$\\
\hline
\end{tabular} 
$$
\caption{Generalized McKay correspondence}
\label{Tab1}
\end{table}
In Table~\ref{Tab1} we use the following notation: $\calC_{2n}$ denotes the cyclic group of order $2n$ ($n \geq 1$), $\calD_k$ the binary dihedral group of order $4k$ ($k \geq 2$), $\calT$ the binary tetrahedral group, and $\calO$ the binary octahedral group.  A rational function
$\prod_{k | d} (1-t^k)^{\alpha_k}$ ($\alpha_k \in \ZZ$)
is represented by its {\em Frame shape} $\prod_{k | d} k^{\alpha_k}$.  We have indicated $p_X$ ($P_G(t)=p_X(t^2)$), $\psi_R$, and $\phi_R$. Note that $B_n^\vee=C_n$, $C_n^\vee=B_n$, $F_4^\vee=F_4$, and $G_2^\vee=G_2$.

We want to interpret Theorem~\ref{Thm:St} in terms of singularities. The group $H/G$ acts in a natural way on the quotient $X=\CC^2/G$. The pair $(X, H/G)$ is a singularity with symmetry. In particular, if $H/G \cong \ZZ_2$ then the pair $(X, H/G)$ is a simple boundary singularity in the sense of Arnold \cite{A78}. We shall consider such singularities in the next section.

\section{Boundary singularities}
Let $\overline{f}$ be a singularity of a function on a manifold with boundary. This means that $\overline{f}$ is the germ of a holomorphic function $\overline{f} : (\CC^{n+1},0) \to (\CC,0)$ such that the function $\overline{f}$ has an isolated critical point at zero both in the space $\CC^{n+1}$ and on the subspace $\CC^n$ given by the equation $x_0=0$, where $x_0,x_1, \ldots , x_n$ are the coordinates of $\CC^{n+1}$. Then $\overline{f}$ defines a germ of a function $f : (\CC^{n+1},0) \to (\CC,0)$ invariant under the involution $\sigma: (x_0,x_1, \ldots, x_n) \mapsto (-x_0, x_1, \ldots , x_n)$ by
$$f(x_0,x_1, \ldots , x_n) := \overline{f}(x_0^2, x_1, \ldots , x_n).$$
Let $(X,0)$ be the singularity defined by $f=0$. We call it the {\em ambient singularity} of the boundary singularity $\overline{f}$. The involution $\sigma$ defines  a natural $\ZZ_2$-action on $(X,0)$. From now on we shall assume $n=2$ and we denote the coordinates of $\CC^3$ by $x,y,z$.

We recall some facts about Coxeter-Dynkin diagrams and monodromy of boundary singularities from \cite{AGV88}.  The middle homology group $H_2(X_\lambda)=H_2(X_\lambda;\ZZ)$ of a Milnor fibre $X_\lambda$ of $(X,0)$ splits as an orthogonal direct sum
$$H_2(X_\lambda) = H^+ \oplus H^-,$$
where $H^+$ and $H^-$ are the subgroups consisting of invariant and antiinvariant cycles respectively with respect to $\sigma_\ast : H_2(X_\lambda) \to H_2(X_\lambda)$. 
There exists a basis $\delta^{(1)}_1, \ldots, \delta^{(1)}_{\mu_0}, \delta^{(2)}_1, \ldots , \delta^{(2)}_{\mu_0}, \delta'_1, \ldots, \delta'_{\mu_1}$ of vanishing cycles of $H_2(X_\lambda)$ such that 
$$\sigma_\ast(\delta^{(1)}_i) = \delta^{(2)}_i, \ i=1, \ldots , \mu_0; \quad \sigma_\ast(\delta'_j)=-\delta'_j, \ j=1, \ldots , \mu_1,$$
and the intersection matrix with respect to this basis has the form
$$\left( \begin{array}{ccc} A & 0 & B \\
                                              0  & A & -B \\
                                              B^t   & -B^t  & A' \end{array} \right)$$
where $A$ and $A'$ are the intersection matrices with respect to the elements $\delta^{(1)}_1, \ldots, \delta^{(1)}_{\mu_0}$ and $\delta'_1, \ldots, \delta'_{\mu_1}$ respectively and $B$ is the matrix 
$(\langle  \delta^{(1)}_i, \delta'_j \rangle)_{i=1, \ldots \mu_0}^{j=1, \ldots , \mu_1}$ ($\langle \cdot , \cdot \rangle$ denotes the intersection form on $H_2(X_\lambda)$). 
Set
$$\widehat{\delta}_i := \delta^{(1)}_i - \delta^{(2)}_i, \quad i=1, \ldots , \mu_0.$$
Then the sublattice $H^-$ of $H_2(X_\lambda)$ spanned by $\widehat{\delta}_1, \ldots ,\widehat{\delta}_{\mu_0}, \delta'_1, \ldots, \delta'_{\mu_1}$ can be considered as the Milnor lattice of the boundary singularity $\overline{f}$. 
The Coxeter-Dynkin diagram of the boundary singularity $\overline{f}$ corresponding to this basis is obtained by {\em folding} the Coxeter-Dynkin diagram corresponding to the basis  $\delta^{(1)}_1, \ldots, \delta^{(1)}_{\mu_0}, \delta^{(2)}_1, \ldots , \delta^{(2)}_{\mu_0}, \delta'_1, \ldots, \delta'_{\mu_1}$ (see Figure \ref{Fig2}). 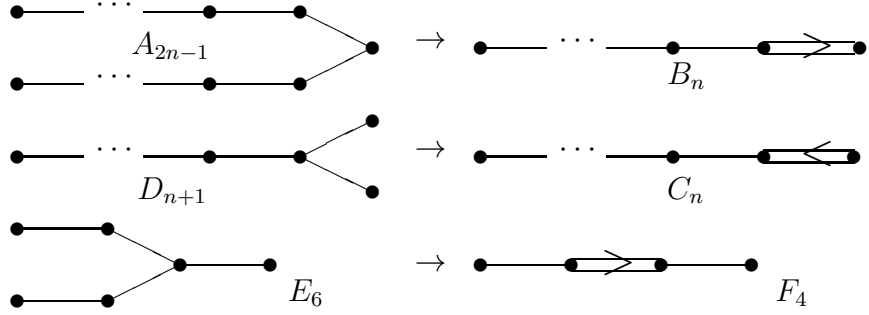
\begin{figure}
$$
\unitlength0.8cm
\begin{picture}(15,5) 
\put(0.1,3.6){\circle*{0.2}}
\put(0.2,3.6){\line(1,0){1}}
\put(1.4,3.6){$\cdots$}
\put(2.2,3.6){\line(1,0){1}}
\put(3.3,3.6){\circle*{0.2}}
\put(3.4,3.6){\line(1,0){1.3}}
\put(4.8,3.6){\circle*{0.2}}
\put(2,4.1){$A_{2n-1}$}
\put(0.1,4.8){\circle*{0.2}}
\put(0.2,4.8){\line(1,0){1}}
\put(1.4,4.8){$\cdots$}
\put(2.2,4.8){\line(1,0){1}}
\put(3.3,4.8){\circle*{0.2}}
\put(3.4,4.8){\line(1,0){1.3}}
\put(4.8,4.8){\circle*{0.2}}
\put(4.8,3.6){\line(2,1){1.2}}
\put(4.8,4.8){\line(2,-1){1.2}}
\put(6,4.2){\circle*{0.2}}
\put(6.7,4.2){$\rightarrow$}
\put(7.8,4.2){\circle*{0.2}}
\put(7.9,4.2){\line(1,0){1}}
\put(9.1,4.2){$\cdots$}
\put(9.9,4.2){\line(1,0){1}}
\put(11,4.2){\circle*{0.2}}
\put(11.1,4.2){\line(1,0){1.3}}
\put(12.5,4.2){\circle*{0.2}}
\put(12.5,4.1){\line(1,0){1.5}}
\put(12.5,4.3){\line(1,0){1.5}}
\put(13.1,4.05){\Large $>$}
\put(14.1,4.2){\circle*{0.2}}
\put(10.9,3.6){$B_n$}
\put(0.1,2.4){\circle*{0.2}}
\put(0.2,2.4){\line(1,0){1}}
\put(1.4,2.4){$\cdots$}
\put(2.2,2.4){\line(1,0){1}}
\put(3.3,2.4){\circle*{0.2}}
\put(3.4,2.4){\line(1,0){1.3}}
\put(4.8,2.4){\circle*{0.2}}
\put(4.8,2.4){\line(2,1){1.2}}
\put(4.8,2.4){\line(2,-1){1.2}}
\put(6,3){\circle*{0.2}}
\put(6,1.8){\circle*{0.2}}
\put(2.1,1.7){$D_{n+1}$}
\put(6.7,2.4){$\rightarrow$}
\put(7.8,2.4){\circle*{0.2}}
\put(7.9,2.4){\line(1,0){1}}
\put(9.1,2.4){$\cdots$}
\put(9.9,2.4){\line(1,0){1}}
\put(11,2.4){\circle*{0.2}}
\put(11.1,2.4){\line(1,0){1.3}}
\put(12.5,2.4){\circle*{0.2}}
\put(12.5,2.3){\line(1,0){1.5}}
\put(12.5,2.5){\line(1,0){1.5}}
\put(13.1,2.25){\Large $<$}
\put(14,2.4){\circle*{0.2}}
\put(10.9,1.7){$C_n$}
\put(0.1,0){\circle*{0.2}}
\put(0.2,0){\line(1,0){1.3}}
\put(1.6,0){\circle*{0.2}}
\put(0.1,1.2){\circle*{0.2}}
\put(0.2,1.2){\line(1,0){1.3}}
\put(1.6,1.2){\circle*{0.2}}
\put(1.6,0){\line(2,1){1.2}}
\put(1.6,1.2){\line(2,-1){1.2}}
\put(2.8,0.6){\circle*{0.2}}
\put(2.8,0.6){\line(1,0){1.5}}
\put(4.3,0.6){\circle*{0.2}}
\put(4.6,0){$E_6$}
\put(6.7,0.5){$\rightarrow$}
\put(7.8,0.6){\circle*{0.2}}
\put(7.9,0.6){\line(1,0){1.3}}
\put(9.3,0.6){\circle*{0.2}}
\put(9.3,0.5){\line(1,0){1.5}}
\put(9.3,0.7){\line(1,0){1.5}}
\put(9.8,0.45){\Large $>$}
\put(10.8,0.6){\circle*{0.2}}
\put(10.8,0.6){\line(1,0){1.5}}
\put(12.3,0.6){\circle*{0.2}}
\put(12.7,0){$F_4$}
\end{picture} 
$$
\caption{Coxeter-Dynkin diagrams of $B_n$, $C_n$, and $F_4$ obtained by folding}
\label{Fig2}
\end{figure}
Note that
$$\langle \widehat{\delta}_i , \widehat{\delta}_i  \rangle = -4, \quad \langle \widehat{\delta}_i , \widehat{\delta}_j \rangle = 2 \langle \delta^{(1)}_i , \delta^{(1)}_j \rangle  \mbox{ for } 1 \leq i,j \leq \mu_0, i \neq j,$$
$$\langle \widehat{\delta}_i, \delta'_j \rangle = 2 \langle \delta^{(1)}_i, \delta'_j \rangle \mbox{ for } 1 \leq i \leq \mu_0, 1 \leq j \leq \mu_1.$$
Therefore the reflection $s_{\widehat{\delta}_i}$ corresponding to an element $\widehat{\delta}_i$ ($1 \leq i \leq \mu_0$), 
$$s_{\widehat{\delta}_i}(x) = x - \frac{2 \langle x, \widehat{\delta}_i \rangle}{\langle \widehat{\delta}_i, \widehat{\delta}_i \rangle} \widehat{\delta}_i \quad \mbox{for } x \in H^-,
$$
is a well defined operator $s_{\widehat{\delta}_i} : H^- \to H^-$.
The Coxeter element (product of reflections) corresponding to the basis $\widehat{\delta}_1, \ldots ,\widehat{\delta}_{\mu_0}, \delta'_1, \ldots, \delta'_{\mu_1}$ of $H^-$ can be considered as the monodromy of the boundary singularity.
Let $\phi_{\overline{f}}$ be its characteristic polynomial. Let $\phi_1$ be the characteristic polynomial of the Coxeter element corresponding to $\widehat{\delta}_1, \ldots ,\widehat{\delta}_{\mu_0}$. 

\begin{proposition} \label{Prop:bound}
One has
$$\phi_{\overline{f}} = \frac{\phi_X}{\phi_1}.$$
\end{proposition}

\begin{proof} 
Write $A = - V^t -V$, $A'= -V'^t - V'$ for upper triangular matrices $V$, $V'$ with diagonal entries equal to  1. Then the formula of \cite[Chap.~V, \S 6, Exercice 3]{B68} for the characteristic polynomial of a Coxeter element yields
\begin{eqnarray*}
\phi_X(t) & = & \left| \begin{array}{ccc} -V^t - tV & 0 & tB \\
                                                                      0   & -V^t - tV & -tB \\
                                                                     B^t &  -B^t & -V'^t -tV' \end{array} \right| \\
& = & \det (-V^t - tV) \left| \begin{array}{cc} -V^t - tV & -2tB \\ -B^t & -V'^t -tV' \end{array} \right| = \phi_1(t) \phi_{\overline{f}}(t).
\end{eqnarray*}
\end{proof}

If $\overline{f}$ is a boundary singularity  then there is a dual boundary singularity  defined by the Lagrange transform $\overline{f}^\vee$ of $\overline{f}$. This is defined as follows (see, e.g., \cite{A93}). Let $\widetilde{f}(x,y,z,p) = px+ \overline{f}(x,y,z)$ be the function in an additional variable $p$ with the equation of the boundary being $p=0$. Then $\widetilde{f}$ can be written as $\widetilde{f}=\overline{f}^\vee(x',y'z')+w^2$ for a function $\overline{f}^\vee$ with an isolated singularity on the boundary $x'=0$. S.~Chmutov and I.~Scherback \cite{CS94} have proved the following theorem.

\begin{theorem}[Chmutov, Scherback] \label{Thm:CS}
Lagrange dual boundary singularities have dual Coxeter-Dynkin diagrams (in the sense that the arrows are reverted). 
\end{theorem}

\begin{corollary} \label{Cor:dual}
The characteristic polynomials of the monodromy of dual boundary singularities are equal.
\end{corollary}

\begin{proof} One can easily see that the reversion of arrows in the Coxeter-Dynkin diagram does not change the characteristic polynomial of the corresponding Coxeter element.
\end{proof}

\section{Simple and unimodal boundary singularities}
The simple and unimodal boundary singularities have been classified by Arnold and V.~I.~Matov. The simple boundary singularities are the singularities $B_n, C_n, F_4$ \cite{A78} where the notation is derived from a  corresponding Coxeter-Dynkin diagram. The singularities $B_n$ are dual to the singularities $C_n$ and $F_4$ is self-dual.

\begin{table}[h]
$$
\begin{tabular}{|c|c|c|c|c|c|}
\hline
 Name & Ambient & $p_X$ & $\psi_{\overline{f}}$  & $\phi_{\overline{f}}$ & Dual\\
\hline
$B_n$ & $A_{2n-1}$ & $2n/1 \cdot n \cdot n$ & $1 \cdot n$ & $2n/n$ & $C_n$\\
$C_n$ & $D_{n+1}$ & $2n/2 \cdot (n-1) \cdot n$ & $2 \cdot (n-1)$ & $2n/n$ & $B_n$ \\
$F_4$ & $E_6$ & $12/3 \cdot 4 \cdot 6$ & $2 \cdot 3$ & $2 \cdot 12/4 \cdot 6$ & $F_4$\\
\hline
$F_{1,0}$ & $\widetilde{E}_8$ & $6/1 \cdot 2 \cdot 3$ &  & $3 \cdot  3$ & $L_6=D_{4,1}$ \\
$K_{4,2}$ & $\widetilde{E}_7$ & $4/1 \cdot 1 \cdot 2$ &  & $2 \cdot  4$ & $K_{4,2}$\\
$L_6=D_{4,1}$ & $\widetilde{E}_6$ & $3/1 \cdot 1 \cdot 1$ &  & $3 \cdot  3$ & $F_{1,0}$\\
\hline
$F_8$ & $E_{14}$ & $24/3 \cdot 8 \cdot 12$ & $3 \cdot 4$ & $4 \cdot 24/8 \cdot 12$ & $E_{6,0}$\\
$F_9$ & $J_{3,0}$ & $18/2 \cdot  6 \cdot  9$ & $2 \cdot 6$ & $18/9$ & $E_{7,0}$\\
$F_{10}$ & $E_{18}$ & $30/3 \cdot  10 \cdot  15$ &  & $15/5$ & $E_{8,0}$\\
$K_8^*$ & $W_{13}$ & $16/ 3 \cdot 4 \cdot  8$ & $3 \cdot  4$ & $16/8$ & $D_5^1$\\
$K_9^*$ & $W_{1,0}$ & $12/2 \cdot  3 \cdot  6$ &  $2 \cdot 6$ & $12/3$ & $E_{6,1}$ \\
$K_8^{**}$ & $W_{12}$ & $20/4 \cdot  5 \cdot  10$ & $2 \cdot  5$ & $2 \cdot  20/4 \cdot  10$ & $K_8^{**}$ \\
$E_{6,0}$ & $Q_{10}$ & $24/6 \cdot  8 \cdot 9$ &  & $4 \cdot 24/8 \cdot 12$ & $F_8$\\
$E_{7,0}$ & $Q_{11}$ & $18/4 \cdot  6 \cdot  7$ &  & $18/9$ & $F_9$\\
$E_{8,0}$ & $Q_{12}$ & $15/ 3 \cdot  5 \cdot  6$ & $3 \cdot  6$ & $15/5$ & $F_{10}$\\
$D_5^1$ & $S_{11}$ &   $16/4 \cdot  5 \cdot 6$ & & $16/8$ & $K_8^*$\\
$E_{6,1}$ & $U_{12}$ & $12/3 \cdot  4 \cdot  4$ & $4 \cdot  4$ & $12/3$ & $K_9^*$ \\
$D_4^2$ & $U_{12}$ & $12/3 \cdot  4 \cdot  4$ & & $4 \cdot  12/2 \cdot  6$ & $D_4^2$ \\
\hline
\end{tabular} 
$$
\caption{Simple, parabolic, and exceptional unimodal boundary singularities}
\label{Tab2}
\end{table}

\begin{table}
$$
\begin{tabular}{|c|c|c|c|c|c|c|}
\hline
 Name & Ambient &  Gabrielov & Dolgachev & $\psi_{\overline{f}}$ & Arnold's dual & Dual\\
\hline
$F_8$ & $E_{14}$ & $2\ 3\ 9$   & $3\ 3\ 4$ & $3 \cdot 4$ & $Q_{10}$ & $E_{6,0}$\\
$F_9$ & $J_{3,0}$ & $2\ 3\ 10$  & $2\ 2\ 2\ 3$ & $2 \cdot 6$  &   & $E_{7,0}$\\
$F_{10}$ & $E_{18}$ & &  &  & & $E_{8,0}$\\
$K_8^*$ & $W_{13}$ & $2\ 5\ 6$ & $3\ 4\ 4$ & $3 \cdot 4$ & $S_{11}$  & $D_5^1$\\
$K_9^*$ & $W_{1,0}$ & $2\ 6\ 6$ & $2\ 2\ 3\ 3$ & $2 \cdot 6$ &  & $E_{6,1}$ \\
$K_8^{**}$ & $W_{12}$ & $2\ 5\ 5$ & $2\ 5\ 5$ & $2 \cdot 5$ & $W_{12}$  & $K_8^{**}$ \\
$E_{6,0}$ & $Q_{10}$ & $3\ 3\ 4$ & $2\ 3\ 9$ & & $E_{14}$ & $F_8$\\
$E_{7,0}$ & $Q_{11}$ & $3\ 3\ 5$ & $2\ 4\ 7$ & & $Z_{13}$ & $F_9$\\
$E_{8,0}$ & $Q_{12}$ & $3\ 3\ 6$ & $3\ 3\ 6$ & $3 \cdot 6$ &  $Q_{12}$  & $F_{10}$\\
$D_5^1$ & $S_{11}$ & $3\ 4\ 4$ & $2\ 5\ 6$ & & $W_{13}$  & $K_8^*$\\
$E_{6,1}$ & $U_{12}$ & $4\ 4\ 4$ & $4\ 4\ 4$ & $4 \cdot 4$ & $U_{12}$ & $K_9^*$ \\
$D_4^2$ & $U_{12}$ & $4\ 4\ 4$ & $4\ 4\ 4$ & & $U_{12}$ & $D_4^2$ \\
\hline
\end{tabular} 
$$
\caption{Gabrielov and Dolgachev numbers of the ambient singularities}
\label{Tab3}
\end{table}

As a corollary of Theorem~\ref{thmhyp} and Proposition~\ref{Prop:bound} we obtain the following reformulation of Theorem~\ref{Thm:St} in the cases $B_n$, $C_n$, and $F_4$. Let $\overline{f}$ be a simple boundary singularity with ambient singularity $(X,0)$. Let $\psi_X$ be the characteristic polynomial of the Coxeter element of the affine root system corresponding to $(X,0)$ and let $\displaystyle{\phi_1= \frac{\phi_X}{\phi_{\overline{f}}}}$.  Define $\displaystyle{\psi_{\overline{f}}:=  \frac{\psi_X}{\phi_1}}$.

\begin{corollary} If $\overline{f}$ is a simple boundary singularity then the Poincar\'e series of the ambient singularity $(X,0)$ satisfies
$$p_X = \frac{\phi_{\overline{f}}}{\psi_{\overline{f}}}.$$
\end{corollary}

According to Matov \cite{M80} the weighted homogeneous unimodal boundary singularities fall into two classes. There are 3 parabolic ones: a dual pair $F_{1,0} \leftrightarrow L_6=D_{4,1}$ and a self-dual one denoted by $K_{4,2}$. Moreover, there are 12 exceptional ones, see Table~\ref{Tab2}. 
The corresponding ambient singularities are either unimodal (8 different singularities, one occuring twice) or bimodal (3 cases). There are 5 dual pairs and 2 self-dual singularities. Looking at Table~\ref{Tab2} one observes that if the ambient singularities are dual in Arnold's sense then the corresponding boundary singularities are Lagrange dual. The Poincar\'e series $p_X$ and the characteristic polynomials $\phi_{\overline{f}}$ are also indicated in Table~\ref{Tab2}. 

\begin{theorem} For 7 of the 12 exceptional unimodal boundary singularities, the Poincar\'e series of the ambient singularity $(X,0)$ is a quotient
$$p_X= \frac{\phi_{\overline{f}}}{\psi_{\overline{f}}}$$
where $\psi_{\overline{f}}$ is the characteristic polynomial of the Coxeter element of a folded $T_{p,q,q}$- or $T_{p,p,q}$-diagram having one vertex less than the Coxeter-Dynkin diagram of $\overline{f}$. \end{theorem}

\begin{proof} For $F_8, K_8^*, K_8^{**}, E_{8,0}$, and $E_{6,1}$ this follows from Theorem~\ref{thmhyp}  by folding the Dolgachev graph of the ambient singularity (cf.\ Table~\ref{Tab3}).

For $K_9^*$ the  ambient singularity is $W_{1,0}$. It has a Coxeter-Dynkin diagram which is an extension of the $T_{2,6,6}$-diagram  \cite{E83}. The statement for this case can be derived from this fact.

For $F_9$ the polynomial $\psi_{\overline{f}}$ is the same as the corresponding polynomial of $K_9^*$.
\end{proof}

\bigskip
\noindent Leibniz Universit\"{a}t Hannover, Institut f\"{u}r Algebraische Geometrie,\\
Postfach 6009, D-30060 Hannover, Germany \\
E-mail: ebeling@math.uni-hannover.de\\

\end{document}